\documentclass{gthack}   
\usepackage[margin=1.4in, left=1.35in, right=1.1in]{geometry}

\usepackage{graphicx}
\usepackage{enumitem}
\usepackage{tikz-cd}

\usepackage{comment}
\usepackage[margin=1cm]{caption}

\usepackage{titlesec}

\title{New curiosities in the menagerie of corks}

\author[Hayden]{Kyle Hayden} 
\address{Columbia University, New York, NY}
\email{hayden@math.columbia.edu}

\author[Piccirillo]{Lisa Piccirillo} 
\address{MIT, Cambridge, MA}
\email{piccirli@mit.edu}

\theoremstyle{plain}

\newtheorem*{ques*}{Question}

\newtheorem{mainthm}{Theorem}[section]

\theoremstyle{definition}

       \newtheorem*{rem*}{Remark}             
\renewcommand{\S}{\textsection}

\usepackage{caption}

\begin{document}

\begin{abstract}
A cork is a smooth, contractible, oriented, compact 4-manifold $W$ together with a self-diffeomorphism $f$ of the boundary 3-manifold that cannot extend to a self-diffeomorphism of $W$; the cork is said to be strong if $f$ cannot extend to a self-diffeomorphism of \emph{any} smooth integer homology ball bounded by $\partial W$. Surprising recent work of Dai, Hedden, and Mallick showed that most of the well-known corks in the literature are strong. We construct the first non-strong corks, which also give rise to new examples of absolutely exotic Mazur manifolds. Additionally we give the first examples of corks where the diffeomorphism of $\partial W$ can be taken to be orientation-reversing.
\end{abstract}

\maketitle
\vspace{-20 pt}

A \emph{cork} is a smooth, contractible, oriented, compact 4-manifold $W$ together with a self-diffeomorphism $f$ of the boundary 3-manifold that cannot extend to a self-diffeomorphism of $W$. Corks are of particular interest in 4-manifold topology because any pair of simply connected \emph{exotic} 4-manifolds $X$ and $X'$ (i.e.~manifolds which are homeomorphic but not diffeomorphic) can be related by removing an embedded cork $W$ from $X$ and regluing it via $f$  to obtain $X'$ \cite{CFHS,matveyev}. Since the pursuit of new exotic phenomena is a primary goal of  
4-manifold topology, so too is the search for new phenomena in the theory of corks. 

 A cork $(W,f)$ is said to be \emph{strong} if $f$ cannot be extended across \emph{any} integer homology ball that $\partial W$ might bound. Extending a 3-manifold self-diffeomorphism over a 4-manifold it bounds depends inherently on the topology of the particular 4-manifold. It is therefore surprising that many of the corks in the existing literature turn out to be strong; Lin, Ruberman, and Saveliev originally showed that the Akbulut cork is strong \cite{lrs}, and Dai, Hedden, and Mallick recently exhibited families of strong corks which encompass most of the corks in the literature \cite{dhm:corks}. In light of this, Dai, Hedden, and Mallick ask if every cork is strong \cite[Question~1.14]{dhm:corks}, expecting a negative answer. 

\begin{mainthm}\label{thm:main}
Not all corks are strong.
\end{mainthm}

In the proof of Theorem~\ref{thm:main}, we give a  straightforward recipe for building a non-strong cork out of any reasonably nice Mazur-type manifold you already have on hand. Notably, our proof requires no new obstructive calculations. We then present an explicit family of examples in Figure~\ref{fig:family}; these corks have some additional interesting properties we record as Theorem \ref{thm:technical}. To state the theorem, we recall that a contractible smooth 4-manifold is \emph{Mazur-type} if it can be described with a single 0-, 1-, and 2-handle; these are the simplest compact, contractible, smooth 4-manifolds other than $B^4$. 

\begin{mainthm}\label{thm:technical}
For integers $|n| \gg 0$, the 4-manifolds $W_n$ and $W_n'$ in Figure~\ref{fig:family} are (absolutely) exotic Mazur-type manifolds, and the indicated boundary involution $f$ extends to a diffeomorphism of $W_n$ but not to a diffeomorphism of $W_n'$.
\end{mainthm}

The methods we develop in the proof of Theorem \ref{thm:main} can be naturally modified to produce corks exhibiting other new phenomena:

\begin{mainthm}\label{thm:achiral}
There exist infinitely many contractible, oriented 4-manifolds $W$ such that $W$ is homeomorphic but not diffeomorphic to $-W$.
\end{mainthm}

If such a 4-manifold $W$ is embedded in a closed, oriented 4-manifold $X$, then one can produce a potentially exotic copy of $X$ by cutting out $W$, reversing its orientation, and regluing $W$ along the orientation-reversing boundary diffeomorphism. We hope that the question of whether such a twist can indeed produce an exotic copy of $X$ (and what special properties such an exotic pair might have)  will prompt 
 further study of these \emph{orientation-reversing corks}.

The proofs of all three theorems are technically extremely straightforward once the reader is equipped with our recipes and any \emph{reasonably nice Mazur manifold}, defined as follows: Let $L$ be a 2-component link of unknots with linking number one. By viewing one component as a dotted circle and the other as the attaching curve of a zero-framed 2-handle, we obtain a handle diagram for a Mazur-type manifold. We say that a Mazur manifold $C$ is \emph{reasonably nice} if it admits such a handle diagram in which the meridian of the dotted circle defines a knot $\mu \subset \partial C$ that does not bound a smoothly embedded disk in $C$. The Akbulut cork \cite{akbulut:cork}, presented on the left in Figure~\ref{fig:akbulut}, is reasonably nice \cite{akbulut-matveyev}, as are many other Mazur manifolds in the literature.
\begin{figure}\center
\def\svgwidth{.75\linewidth}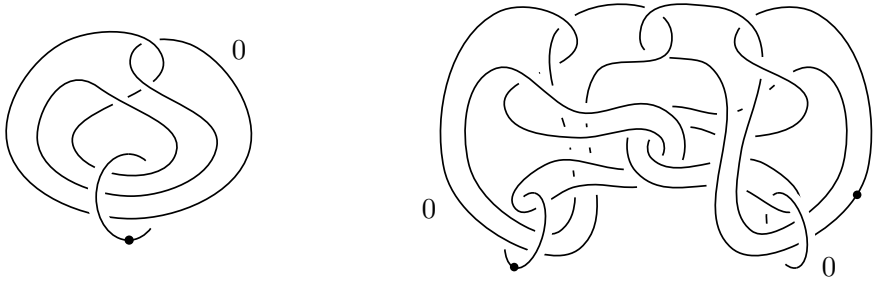
\caption{(Left) The Akbulut cork, which is strong. (Right) A cork that is not strong.}\label{fig:akbulut}
\end{figure}

\vspace{-10pt}
\begin{proof}[Proof of Theorem 1]  Choose any reasonably nice Mazur manifold $C$ with a diagram as above; label the dotted unknot $j$ and the 2-handle curve $h$.  Draw this diagram on the left side of the page, and let $C'$ be a second copy of $C$ (with corresponding dotted unknot $j'$ and 2-handle curve $h'$), drawn on the right side of the page by rotating the diagram for $C$ by $180^\circ$ through the vertical line down the center of the page. Then modify the linking between the 2-handles of $C$ and $C'$ in any manner you like so long as your diagram retains this rotational symmetry. You are, however, forbidden from introducing any new geometric linking between any 1-handle and 2-handle. If desired, you may do nothing at all and continue on with $C \natural C'$. This is your first contractible 4-manifold, denoted $W$, and it admits a smooth involution $F$ arising from the obvious symmetry of its diagram.

Observe that $Y=\partial W$ bounds another contractible 4-manifold  $W'$ obtained by reversing the roles of the 1- and 2-handle curves in $C'$, i.e.~putting a dot on $h'$ and letting $j'$ represent a zero-framed 2-handle.  We claim the involution $f=F|_{\partial W}: Y \to Y$ does not extend to a diffeomorphism of $W'$, hence that $(W', f)$ is a non-strong cork. (Note that $f$ extends to a \emph{homeomorphism} of the contractible 4-manifold $W'$ by work of Freedman \cite{freedman}.)

Let $\mu \subset Y$ denote the meridian of the 1-handle curve $j$ in $C \subset W'$. Since $f(\mu)$ is represented by the meridian of the 2-handle $j'$, $f(\mu)$ is slice in $W'$. If $f$ were to extend over $W'$, then $\mu$ would also have to be smoothly slice in $W'$.  To show that this is not the case, consider the 4-manifold $X$ whose diagram is obtained from that of $W'$ by erasing the dotted circle $h'$ on the right side of the page.  Note that $W'$ is obtained from $X$ by carving out a slice disk represented by the dotted circle $h'$, and hence that $W'$ embeds in $X$. Therefore if $\mu$ is slice in $W'$ then the curve in $\partial X$ represented by $\mu$ in the diagram (which we'll abusively still call $\mu$) is slice in $X$. Observe also that $X$ is obtained from $C$ by attaching a zero-framed 2-handle along $j'$. By construction, $j'$ is an unknot split from the handle diagram of $C$, so any disk passing over this 2-handle can be replaced with one disjoint from it. Therefore if $\mu$ is slice in $X$, then $\mu$ is slice in $C$, a contradiction. 
\end{proof}

\begin{rem*}
This recipe can be modified to the reader's taste; for example, it is possible to use a more complicated base cork $C$ or to produce a higher-order boundary diffeomorphism.
\end{rem*}
\vspace{-10pt}
\begin{figure}\center
\def\svgwidth{\linewidth}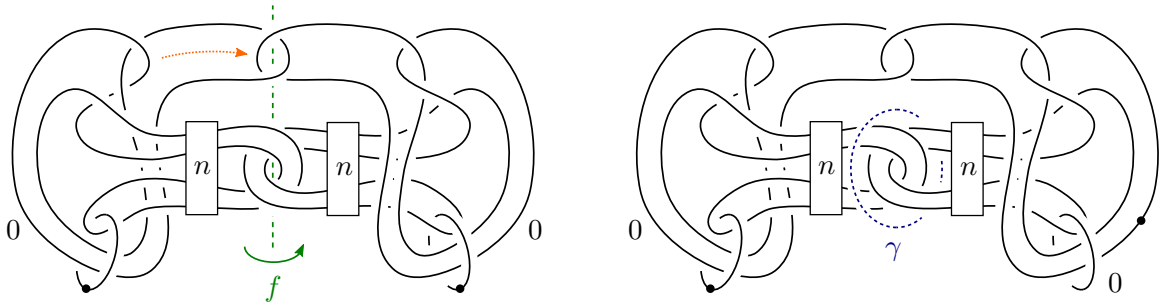
\caption{Contractible 4-manifolds $W_n$ (left) and $W_n'$ (right). Boxes denote full twists.} \label{fig:family}
\end{figure}
\begin{proof}[Proof of Theorem \ref{thm:technical}]
For all integers $n$, the manifolds $W_n$ and $W_n'$ in Figure~\ref{fig:family} are produced via the recipe outlined above, so the obvious boundary involution $f$ extends smoothly over $W_n$ and not over $W_n'$. Thus it remains to show that $W_n$ and $W_n'$ are both Mazur-type for all $n$, and that  the pair are homeomorphic but not diffeomorphic for $|n|\gg0$.  

To see that $W_n'$ is Mazur-type, observe that the 2-component link $h\cup h'$ is a Hopf link, thus the corresponding 1- and 2-handles form a canceling pair. To see that $W_n$ is Mazur-type, first perform the handle slide of $h$ indicated by the orange arrow. In the new diagram, this modified 2-handle and the 1-handle represented by $j$ form a canceling pair.

Observe that $\partial W_n$ and $\partial W_n'$ are diffeomorphic, each identified with the 3-manifold $Y_n$ obtained by performing zero-framed Dehn surgery on all four link components in the diagram of $W_n$ or $W_n'$. Since $W_n$ and $W_n'$ are contractible 4-manifolds with the same boundary, they are homeomorphic by work of Freedman \cite{freedman}. For $|n| \gg0$, we will show that there are exactly two self-diffeomorphisms of $Y_n$ up to isotopy, namely $f$ and the identity. Assuming this for the moment, we show that $W_n$ and $W_n'$ (for $|n|\gg0$) are not diffeomorphic: To the contrary, suppose there exists a diffeomorphism $\phi$ of $Y_n$ extending to a diffeomorphism from $W_n'$ to $W_n$. Since $f$ extends over $W_n$, the composition $\phi^{-1} \circ f \circ \phi$ extends to a diffeomorphism of $W_n'$. However, $\phi^{-1} \circ f \circ \phi$ cannot be isotopic to the identity (because the isotopy class of the identity is preserved under conjugation), so it must be isotopic to $f$. This implies that $f$ extends to a diffeomorphism of $W_n'$, a contradiction.

It remains to show that, for $|n| \gg 0$, there are exactly two self-diffeomorphisms of $Y_n$ up to isotopy, namely $f$ and the identity. Since $f$ fails to extend to a diffeomorphism of $W_n'$, it is not isotopic to the identity. So it suffices to show that the mapping class group of $Y_n$ has no more than two elements. To do so, we realize $Y_n$ as $-1/2n$-surgery on the knot $\gamma \subset Y_0$ depicted in Figure~\ref{fig:family}. Using SnapPy \cite{snappy} and Sage \cite{sagemath}, we verify that $\gamma \subset Y_0$ has hyperbolic exterior; all computer calculations for this proof are documented in \cite{hp:notstrongdoc}.  For large $|n|$, Thurston's hyperbolic Dehn surgery theorem \cite{thurston:notes} ensures that the Dehn-filled 3-manifold $Y_n$ is hyperbolic and that the core $\tilde \gamma \subset Y_n$ of the surgered solid torus is the unique shortest closed geodesic in  $Y_n$.  By Mostow rigidity \cite{mostow}, the mapping class group of a hyperbolic 3-manifold is isomorphic to its isometry group.   As in \cite[\S5]{kojima}, we note that any isometry of $Y_n$ fixes the short geodesic $\tilde \gamma$ setwise, hence $\operatorname{Isom}(Y_n)$ is isomorphic to a subgroup of $\operatorname{Isom}(Y_n \setminus \tilde \gamma)$.  Since the hyperbolic 3-manifolds $Y_n \setminus \tilde \gamma$ and $Y_0 \setminus \gamma$ are naturally identified, we have $\operatorname{Isom}(Y_n \setminus \tilde \gamma) \cong \operatorname{Isom}(Y_0 \setminus \gamma)$. Using \cite{snappy,sagemath}, we calculate $\operatorname{Isom}(Y_0 \setminus \gamma)\cong \mathbb{Z}_2$, so we conclude that $\operatorname{Isom}(Y_n)$ contains at most (indeed, exactly) two elements.
\end{proof}

\begin{rem*}
For $|n| \gg0$, the exotic 4-manifolds  $W_n$ and $W_n'$ have qualitatively distinct diffeomorphism groups: All diffeomorphisms of $\partial W_n$ extend to diffeomorphisms of $W_n$, whereas a diffeomorphism of $\partial W_n'$ extends over $W_n'$ only if it is isotopic to the identity.
\end{rem*}

\begin{proof}[Proof of Theorem \ref{thm:achiral}]
Let $W_n$ and its mirror $-W_n$ be the contractible 4-manifolds depicted in Figure~\ref{fig:WandMirror}, with $Y_n$ and $-Y_n$ denoting their boundaries.  
 We will begin by demonstrating that for any integer $n$, there is a diffeomorphism $f$ from $Y_n$ to $-Y_n$ that extends to a homeomorphism from $W_n$ to $-W_n$. We will then show that for all $n$, $f$ does not extend to a diffeomorphism from $W_n$ to $-W_n$, and that for $|n|$ sufficiently large, $W_n$ is not diffeomorphic to $-W_n$.

\smallskip

\begin{figure}\center
\def\svgwidth{.9\linewidth}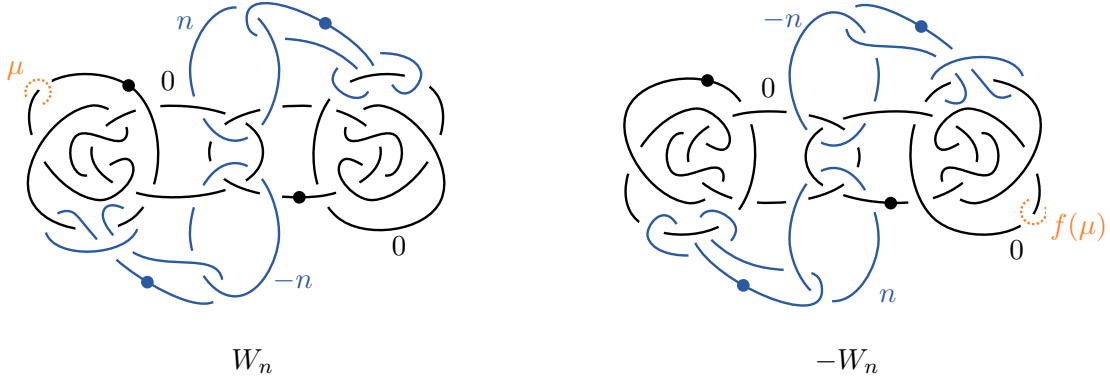
\caption{Contractible manifolds $W_n$ and $-W_n$ from the proof of Theorem~\ref{thm:achiral}.}\label{fig:WandMirror}
\end{figure}

Observe that $Y_n$ has a surgery diagram given by the framed link shown in Figure~\ref{fig:Yn} (where the curves $\ell_1$ and $\ell_2$ may be ignored for the moment).  Rotating this diagram by $180^\circ$ through its center carries the framed link to its mirror, which is a surgery diagram for $-Y_n$. This isotopy of framed links induces a diffeomorphism of 3-manifolds $f:Y_n \to -Y_n$. To extend $f$ to a homeomorphism $W_n \to -W_n$, we rotate the diagram of $W_n$ by $180^\circ$ through its center and then reverse the roles of the black 1- and 2-handle curves, yielding the diagram for $-W_n$. The latter operation corresponds to cork-twisting embedded copies of both the Akbulut cork and its mirror, which induces a homeomorphism supported away from the boundary and thus provides the desired extension of $f$ to a homeomorphism $W_n \to -W_n$.

Now we'll argue that $f$ does not extend to a diffeomorphism. Consider the curves $\mu\subset Y_n$ and $f(\mu)\subset -Y_n$ marked in Figure \ref{fig:WandMirror}. Observe that  $f(\mu)$ bounds a smoothly embedded disk in $-W_n $. Therefore, if $f$ were to extend smoothly, $\mu$ would have to bound a disk in $W_n$. To show that this is not the case, consider the 4-manifold $X_n$ whose diagram is obtained from that of $W_n$ by erasing the black dotted 1-handle on the right side of $W_n$.   Note that $W_n$ is obtained from $X_n$ by carving out a slice disk represented by that dotted circle, and hence that $W_n$ embeds in $X_n$. Therefore, if $\mu$ is slice in $W_n$, then the curve represented by $\mu$ (abusively still called $\mu$) in $\partial X_n$ is slice in $X_n$. Observe now that $X_n$ is diffeomorphic to the boundary connected sum of the  Akbulut cork with $S^2\times D^2$. Therefore, as argued in the proof of Theorem $1$, if $\mu$ is slice in $X_n$, then $\mu$ is slice in the Akbulut cork. Thus $\mu$ is not slice in $W_n$, so $f$ does not extend to a diffeomorphism.

\begin{figure}\center
\def\svgwidth{.4\linewidth}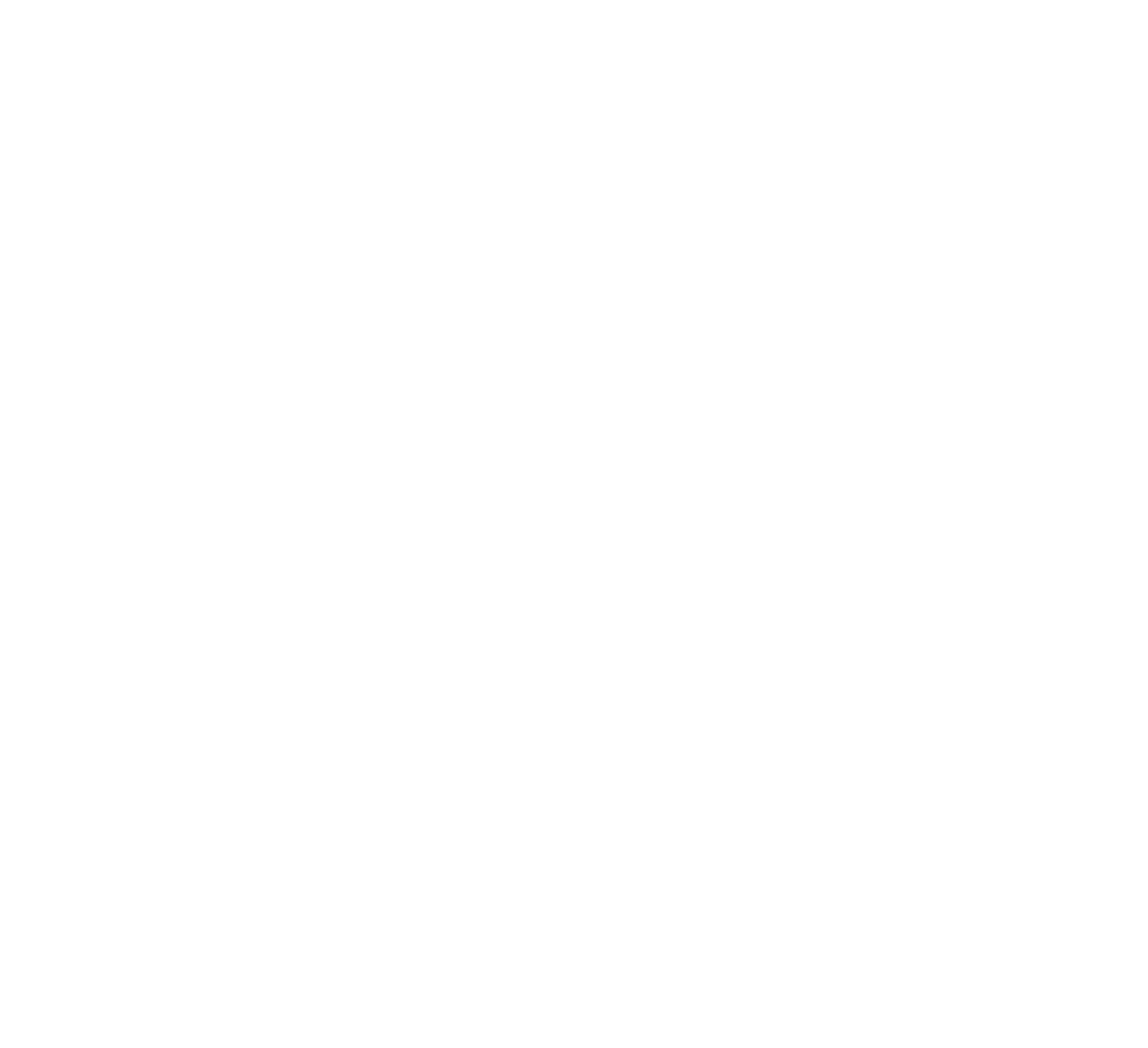
\caption{}\label{fig:Yn}
\end{figure}

It remains to show that for sufficiently large $|n|$, $W_n$ is not diffemorphic to $-W_n$. Observe that any such diffeomorphism must restrict to an orientation-preserving diffeomorphism $\phi:Y_n\to -Y_n$ on the boundary, which can also be viewed as an orientation-reversing self-diffeomorphism of $Y_n$. We will show that, for sufficiently large $|n|$, every orientation-reversing self-diffeomorphism of $Y_n$ is isotopic to the self-diffeomorphism $f$ defined above.  Since $f$ cannot be smoothly extended, there can be no diffeomorphism $W_n \to -W_n$.

We will again classify the mapping class group of $Y_n$ for $|n| \gg 0$, using an argument nearly identical to the one in the proof of Theorem \ref{thm:technical}. First observe that $Y_n$ is obtained from $Y_0$ by $\pm 1/n$-surgery on the link $\ell=\ell_1\cup\ell_2$ in $Y_0$ depicted in Figure~\ref{fig:Yn}. Using SnapPy \cite{snappy} and Sage \cite{sagemath}, we verify that $\ell\subset Y_0$ has hyperbolic exterior; as above, computer calculations for this proof are documented in \cite{hp:notstrongdoc}.  For large $|n|$, Thurston's hyperbolic Dehn surgery theorem \cite{thurston:notes} ensures that the Dehn-filled 3-manifold $Y_n$ is hyperbolic and that the cores $\tilde \ell \subset Y_n$ of the surgered solid tori are the two shortest closed geodesics in  $Y_n$. By Mostow rigidity \cite{mostow}, the mapping class group of a hyperbolic 3-manifold is isomorphic to its isometry group.   As in \cite[\S5]{kojima}, we note that any isometry of $Y_n$ fixes the pair of short geodesics $\tilde \ell$ setwise, hence $\operatorname{Isom}(Y_n)$ is isomorphic to a subgroup of $\operatorname{Isom}(Y_n \setminus \tilde \ell)$.  Since the hyperbolic 3-manifolds $Y_n \setminus \tilde \ell$ and $Y_0 \setminus \ell$ are naturally identified, we have $\operatorname{Isom}(Y_n \setminus \tilde \ell) \cong \operatorname{Isom}(Y_0 \setminus \ell)$. Using \cite{snappy,sagemath}, we calculate $\operatorname{Isom}(Y_0 \setminus \ell)\cong \mathbb{Z}_2$, so we conclude that $\operatorname{Isom}(Y_n)$ contains at most two elements. Indeed it has exactly two: the identity and $f$ (which is orientation-reversing and hence distinct from the identity). Thus, as claimed, $f$ is the only orientation-reversing self-diffeomorphism of $Y_n$.
\end{proof}

\smallskip

\textbf{Acknowledgements} \, We thank Danny Ruberman and Irving Dai for helpful conversations. K.H.~was supported in part by NSF grant DMS-1803584, and L.P.~was supported in part by NSF grant DMS-1902735.

\begingroup
\titleformat*{\section}{\bf \large}
\bibliographystyle{gtart}
\bibliography{biblio}

\begin{thebibliography}{}
\providecommand\bibmarginpar{\leavevmode\marginpar}
\def\urlstyle#1{{\tt #1}}

\bibitem{akbulut:cork}
\textbf{S Akbulut},
  \href{http://projecteuclid.org.proxy.bc.edu/euclid.jdg/1214446320} {\emph{A
  fake compact contractible {$4$}-manifold}}, J. Differential Geom. 33 (1991)
  335--356

\bibitem{akbulut-matveyev}
\textbf{S Akbulut}, \textbf{R Matveyev}, \emph{Exotic structures and adjunction
  inequality}, Turkish J. Math. 21 (1997) 47--53

\bibitem{snappy}
\textbf{M Culler}, \textbf{N\,M Dunfield}, \textbf{M Goerner}, \textbf{J\,R
  Weeks}, \emph{Snap{P}y, a computer program for studying the geometry and
  topology of $3$-manifolds}, \url{http://snappy.computop.org}

\bibitem{CFHS}
\textbf{C Curtis}, \textbf{M Freedman}, \textbf{W Hsiang}, \textbf{R Stong},
  \href{http://eudml.org/doc/144350} {\emph{A decomposition theorem for
  h-cobordant smooth simply-connected compact 4-manifolds.}}, Inventiones
  mathematicae 123 (1996) 343--348

\bibitem{dhm:corks}
\textbf{I Dai}, \textbf{M Hedden}, \textbf{A Mallick}, \emph{Corks,
  involutions, and Heegaard Floer homology}, arXiv preprint arXiv:2002.02326
  (2020)

\bibitem{freedman}
\textbf{M\,H Freedman},
  \href{http://projecteuclid.org.proxy.bc.edu/euclid.jdg/1214437136} {\emph{The
  topology of four-dimensional manifolds}}, J. Differential Geom. 17 (1982)
  357--453

\bibitem{hp:notstrongdoc}
\textbf{K Hayden}, \textbf{L Piccirillo}, \emph{Ancillary files for ``New
  curiosities in the menagerie of corks''}, arXiv (August 2020)

\bibitem{kojima}
\textbf{S Kojima}, \href{http://dx.doi.org/10.1016/0166-8641(88)90027-2}
  {\emph{Isometry transformations of hyperbolic {$3$}-manifolds}}, Topology
  Appl. 29 (1988) 297--307

\bibitem{lrs}
\textbf{J Lin}, \textbf{D Ruberman}, \textbf{N Saveliev}, \emph{On the Froyshov
  invariant and monopole Lefschetz number}, arXiv preprint arXiv:1802.07704
  (2018)

\bibitem{matveyev}
\textbf{R Matveyev}, \href{http://dx.doi.org/10.4310/jdg/1214459222} {\emph{A
  decomposition of smooth simply-connected $h$-cobordant 4-manifolds}}, J.
  Differential Geom. 44 (1996) 571--582

\bibitem{mostow}
\textbf{G\,D Mostow}, \href{http://www.numdam.org/item/PMIHES_1968__34__53_0}
  {\emph{Quasi-conformal mappings in $n$-space and the rigidity of hyperbolic
  space forms}}, Publications Math\'ematiques de l'IH\'ES 34 (1968) 53--104

\bibitem{sagemath}
\textbf{{The Sage Developers}}, \emph{{S}ageMath, the {S}age mathematics
  software system}, \url{https://www.sagemath.org} (2019)

\bibitem{thurston:notes}
\textbf{W\,P Thurston}, \emph{The geometry and topology of three-manifolds},
  Available at \url{http://msri.org/publications/books/gt3m/} (1978)

\end{thebibliography}
\endgroup

\end{document}